\documentclass[reqno]{amsart}
\usepackage{amsmath, amsthm, amssymb, amstext, mathrsfs}

\usepackage[left=3.4cm,right=3.4cm,top=3cm,bottom=3cm]{geometry}
\usepackage{hyperref,xcolor}
\hypersetup{
 pdfborder={0 0 0},
 colorlinks,
}
\usepackage{enumitem}
\allowdisplaybreaks

\newtheorem{theorem}{Theorem}

\newtheorem{lemma}[theorem]{Lemma}

\newtheorem{definition}[theorem]{Definition}

\numberwithin{theorem}{section}
\numberwithin{equation}{section}

\title[Fractional partial differential variational inequality]{Fractional partial differential variational inequality}

\author[Jinxia Cen]{Jinxia Cen}
\address[Jinxia Cen]{School of Mathematical Sciences, Zhejiang Normal University, Jinhua 321004, P. R. China}
\email{jinxcen@163.com}

\author[J. Vanterler da C. Sousa]{J. Vanterler da C. Sousa$^{*}$}
\address[J. Vanterler da C. Sousa]{Aerospace Engineering, PPGEA-UEMA, Department of Mathematics, DEMATI-UEMA, Sao Lu\'is, MA 65054, Brazil.}
\email{vanterler@ime.unicamp.br}
\thanks{$^{*}$Corresponding author}

\author[Wei Wu]{Wei Wu}
\address[Wei Wu]{Center for Applied Mathematics of Guangxi, and Guangxi Colleges and Universities Key Laboratory of Complex System Optimization and Big Data Processing, Yulin Normal University, Yulin 537000, Guangxi, P.R. China}
\email{wwu1985yl@163.com}

\subjclass[2010]{47J20, 35R11, 49J40, 35J88, 26A33}
\keywords{Fractional partial differential variational inequality, Nonlocal boundary condition, Mittag-Leffler functions, Existence, Measure of noncompactness}
\begin{document}

\begin{abstract}
In this present paper, we introduce and study a dynamical systems involving fractional derivative operator and nonlocal condition, which is constituted of a fractional evolution equation and a time-dependent variational inequality, and is named as fractional partial differential variational inequality (FPDVI, for short).
By employing the estimates involving the one-and two-parameter Mittag-Leffler functions, fixed-point theory for set-value mappings, and non-compactness measure theory, we develop a general framework to establish the existence of  smooth solutions to (FPDVI).
\end{abstract}

\maketitle

\bigskip
	
\section{Introduction}

Let $\Omega_{1}, \Omega_{2}$ be two separable and reflexive Banach spaces, and $K$ be a nonempty, closed and convex subset of $\Omega_{1}$. In addition, $\mathcal{A}\colon \Omega_{2} \rightarrow \Omega_{2}$ is a positive sectorial operator of the Mittag-Leffler family $\mathbb{E}_{\alpha}(\xi^{\alpha} \mathcal{A})$ in $\Omega_{2}$ with $\xi\geq 0$.
Let $J$ be a finite interval (i.e., $J=[a,b]$) or infinite interval (i.e., $J=[a,+\infty)$) of the real line $\mathbb{R}$ and $\alpha>0$. The Riemann-Liouville fractional integral (left-sided and right-sided) of a function $\phi$ on $J$ is defined by \cite{J1,frac2}
\begin{equation}\label{eq111}
\mathcal{I}_{a+}^{\alpha }\phi\left( \xi\right) =\frac{1}{\Gamma \left( \alpha \right) }\int_{a}^{\xi}\left(\xi -s \right) ^{\alpha -1}\phi\left( s\right) ds.
\end{equation}
On the other hand, let $n-1< \alpha \leq n$, with $n \in \mathbb{N}$  and $\phi\in \mathcal{C}^n(J,\mathbb{R})$. The Caputo fractional partial derivative denoted by ${^{\rm c}\mathfrak{D}}^{\alpha}_{a+}(\cdot)$ of a function $f$ of order $\alpha$, is defined by \cite{J1,frac2}
\begin{equation}\label{eq333}
{^{\rm c}\mathfrak{D}}^{\alpha}_{a+}\phi(\xi):= \frac{\partial^{\alpha} \phi(\xi)}{\partial t^{\alpha}} =
\mathcal{I}_{a+}^{n-\alpha} \left(\frac{\partial^{\alpha} \phi(\xi)}{\partial t^{\alpha}} \right),
\end{equation}
where $\mathcal{I}_{a+}^{\alpha}(\cdot)$ is given  {\rm Eq.(\ref{eq111})}. A natural consequence of the definition (\ref{eq333}), is that in the limit of $\alpha\rightarrow n$, we have the classical derivative (integer order), given by ${^{\rm c}\mathfrak{D}}^{n}_{a+}\phi(\xi):= \dfrac{\partial^{n} \phi(\xi,t)}{\partial t^{n}}$.

Variational inequalities as useful models have been widely applied to solve the problems with convex energy functions \cite{11,21,32}. An essential generalization of variational inequalities are hemivariational inequalities, which were mastered by Panagiotopoulos \cite{49} in the early 1980s and are mainly connected to engineering applications involving non-monotone and possibly multivalued constitutive and interface laws (see \cite{16,22,55} and the references therein). A wide range of physical phenomena leads to variational and hemivariational inequalities, constituting one of the most promising branches of pure mathematics and applied mathematics. This is the reason why inequality problems in Mechanics could be divided into two main categories: that of variational inequalities, which is mainly concerned with convex functions, and that of hemivariational inequalities, which is concerned with locally non-convex Lipschitz functions.

Recently, the notations of differential variational inequalities and differential hemivariational inequalities were introduced by Pang~\cite{19} and  Zeng~\cite{55}, respectively, for solving various problems in real-life which are modeled by inequalities with constraints that are (partial) differential equations. After that, more and more scholars are attracted to develop the theory, applications and numerical methods for differential variational/hemivariational inequalities \cite{Cen,Liu1,Liu2,Migorski1,Migorski2}. For example, in 2013, Gwimer \cite{5}, investigated the stability of a new class of differential variational inequalities; in 2017, Li et al. \cite{13}, explored the properties of solution set for a systems of nonlinear evolutionary partial differential equations. On the other hand, the theory of fractional operators is well established with numerous results of great relevance and impact in several areas, especially when it involves applications, of which we mention some of them, namely, Engineering, Physics, Medicine, among others \cite{aplica1,aplica3,aplica4} and \cite{J1,novo3,frac2,J2}. During these last years, new fractional operators have appeared, which has contributed a lot to the growth of the area. Here, it should be mentioned that despite, in the present work, we restrict ourselves to Caputo fractional operator, but the results established in the current paper could be extended to different fractional operators.

Because of the important of both differential variational/hemivariational inequalities and fractional calculus, in 2019, Mig\'orski and Zeng \cite{Migorski3}, discussed a fractional partial evolution variational system consisting of a mixed quasi-variational inequality combined with a fractional partial differential equation in the Caputo sense. In \cite{Migorski3}, the authors used pseudo-monotonicity of multivalued operators, a generalization of the Knaster-Kuratowski-Mazurkiewicz theorem, operator semigroup theory and the Bohnenblust-Karlin fixed-point principle to discuss the existence of mild solutions for the system. Since then, some experts move their attention to the study of fractional differential variational inequalities, see \cite{Migorski4,Jiang} and the references therein.

Furthermore, we can highlight the interesting paper on the good placement of a generalized vector variational inequality problem in the structure of topological vector spaces investigated by Kumar and Gupta \cite{para1}.

In 2020 Papageorgiou et al. \cite{para2}, discussed a nonlinear control system involving a maximal monotone map and with a priori feedback. The authors assume that the multi-functional control constraint $U(t,x)$ has non-convex value and only lsc in the variable $x$. Using the $Q$ regularization of $U(t,\cdot)$, the authors introduce a relaxed system and discuss a result that states that any state of the relaxed system can be approximated to the appropriate norm with any degree of accuracy, by a state of the original system. The present work presents interesting results that can be connected with fractional operators, in particular, with the one presented in this paper. On the other hand, we can highlight the interesting work carried out by Cortez et al. \cite{para3}, using fractional integrals with locals and the Mittag-Leffler Kernel to address inequalities of the Hermite-Hadamard type.

Motivated by the works above, in particular by the \cite{principal} work, in this paper, we consider the following fractional partial differential variational inequality with a nonlocal boundary condition given by
\begin{eqnarray}\label{principal}
\left\{
 \begin{array}{crl}
 {^{\rm c}\mathfrak{D}}^{\alpha}_{0+} \theta(\xi)&=& \mathcal{A}\theta(\xi) +\mathcal{B}(\xi,\theta(\xi))u(\xi)+f(\xi,\theta(\xi)), \,\, a.e.\,\,\xi\in [0,T],\\
u(\xi) &\in& SOL(K,g(\xi,\theta(\xi))+G(\cdot),\phi),\,\,a.e.\,\,\xi\in [0,T],\\
\theta(0)&=&h(\theta),
\end{array}
\right.
\end{eqnarray}
where ${^{\rm c}\mathfrak{D}}^{\alpha}_{0+}(\cdot)$ is the Caputo fractional derivative of order $0<\alpha<1$,  $\mathcal{B}:[0,T]\times \Omega_{2} \rightarrow \mathscr{L}(\Omega_{1}, \Omega_{2})$; $f:[0,T]\times \Omega_{2} \rightarrow \Omega_{2}$; $G: K\rightarrow \Omega_{1}^{*}$; $\phi: K\rightarrow \overline{\mathbb{R}}:= \mathbb{R}\cup \left\{+\infty \right\}$ and $h:C([0,T], \Omega_{2})\rightarrow \Omega_{2}$, are given maps, which will be specified in the sequel. Furthermore, $\Omega_{1}^{*}$ is dual space of $\Omega_{1}$ and $SOL(K,g(\xi,\theta(\xi)),G(\cdot),\phi)$ is the solution set of the following generalized mixed variational inequality in $\Omega_{1}$: given $\xi\in[0,T]$ find $u: [0,T]\rightarrow K$ such that
\begin{equation*}
    \left< g(\xi,\theta(\xi))+G(u(\xi)), v-u(\xi) \right>+ \phi(v)- \phi(u(\xi))\geq 0
\end{equation*}
$\forall v\in K$, a.e. $\xi\in [0,T]$.

The mild solution of {\rm(\ref{principal})} is formulated by the following way, see \cite{mild1,Bazhlekova,Cuesta}.

\begin{definition}\label{Definition1.1} A pair of functions $(\theta,u)$ with $\theta\in C([0,T],\Omega_{2})$ and $u:[0,T]\rightarrow K(\subseteq \Omega_{1})$ measurable, is said to be a mild solution of {\rm(\ref{principal})} if all $\xi\in[0,T]$ the equality holds
\begin{eqnarray}\label{1.3}
    \theta(\xi)= \mathbb{E}_{\alpha} (\xi^{\alpha} \mathcal{A}) h(\theta)+\int_{0}^{\xi} (\xi-s)^{\alpha-1} \mathbb{E}_{\alpha,\alpha} ((\xi-s)^{\alpha} \mathcal{A}) (\mathcal{B}(s,\theta(s))u(s)+f(s,\theta(s)))ds,
\end{eqnarray}
where $u(\xi)\in SOL(K,g(\xi,\theta(\xi))+G(\cdot), \phi)$ a.e., $\xi\in [0,T]$. If $(\theta,u)$ is a mild solution of the problem {\rm(\ref{principal})}, then $\theta$ is called to be the mild trajectory and $u$ is the variational control trajectory. Furthermore, we have that $\mathbb{E}_{\alpha}(\cdot)$ and $\mathbb{E}_{\alpha,\alpha}(\cdot)$, are the Mittag-Leffler functions of one and two parameters, respectively.
\end{definition}

The main contribution of this present paper is to discuss the existence of smooth solutions for the fractional partial differential variational inequality (\ref{principal}) via using the estimates involving the Mittag-Leffler functions of one and two parameters, fixed point theory of multivalued mapping, and non-compactness measure theory.

Moreover, to highlight the level of generalization of our problem (\ref{principal}), we present below several its particular cases.
\begin{enumerate}

\item  If $\alpha=1$, then (\ref{principal}) reduces to the following differential variational inequality
    \begin{eqnarray*}
\left\{
 \begin{array}{crl}
 \theta'(\xi)&=& \mathcal{A}\theta(\xi) +\mathcal{B}(\xi,\theta(\xi))u(\xi)+f(\xi,\theta(\xi)), \,\, a.e.\,\,\xi\in [0,T],\\
u(\xi) &\in& SOL(K,g(\xi,\theta(\xi))+G(\cdot), \phi),\,\,a.e.\,\,\xi\in [0,T],\\
\theta(0)&=&h(\theta),
\end{array}
\right.
\end{eqnarray*}
which has been explored by Liu et al. \cite{principal} .

\item  If $\Omega_{1}=\mathbb{R}^{m}$, $\Omega_{2}=\mathbb{R}^{n}$, $f=0$, $\phi=0$ and $G: K \rightarrow \mathbb{R}^{m}$, then (\ref{principal}) reduces to the following fractional differential variational inequality in finite-dimensional spaces
    \begin{eqnarray*}
\left\{
 \begin{array}{crl}
 {^{\rm c}\mathfrak{D}}^{\alpha}_{0+} \theta(\xi)&=& \mathcal{A}\theta(\xi) +\mathcal{B}(\xi,\theta(\xi))u(\xi), \,\, a.e.\,\,\xi\in [0,T],\\
u(\xi) &\in& SOL(K,g(\xi,\theta(\xi))+G(\cdot),\phi),\,\,a.e.\,\,\xi\in [0,T],\\
\theta(0)&=&h(\theta).
\end{array}
\right.
\end{eqnarray*}

 \item  If $\alpha=1$, $\Omega_{1}=\mathbb{R}^{m}$, $\Omega_{2}=\mathbb{R}^{n}$, $f=0$, $\phi=0$ and $G: K \rightarrow \mathbb{R}^{m}$, then (\ref{principal}) reduces to the following differential variational inequality in finite-dimensional spaces
    \begin{eqnarray*}
\left\{
 \begin{array}{crl}
 \theta'(\xi)&=& \mathcal{A}\theta(\xi) +\mathcal{B}(\xi,\theta(\xi))u(\xi), \,\, a.e.\,\,\xi\in [0,T],\\
u(\xi) &\in& SOL(K,g(\xi,\theta(\xi))+G(\cdot)),\,\,a.e.\,\,\xi\in [0,T],\\
\theta(0)&=&h(\theta).
\end{array}
\right.
\end{eqnarray*}

 \item  If $\alpha=1$, $\Omega_{1}=\mathbb{R}^{m}$, $\Omega_{2}=\mathbb{R}^{n}$, $\mathcal{A}=0$ and $G: K \rightarrow \mathbb{R}^{m}$ is a single valued mapping, then (\ref{principal}) reduces to the following differential variational inequality  in finite-dimensional spaces
    \begin{eqnarray*}
\left\{
 \begin{array}{crl}
 \theta'(\xi)&=& \mathcal{B}(\xi,\theta(\xi))u(\xi)+f(\xi,\theta(\xi)), \,\, a.e.\,\,\xi\in [0,T],\\
u(\xi) &\in& SOL(K,g(\xi,\theta(\xi))+G(\cdot),\phi),\,\,a.e.\,\,\xi\in [0,T],\\
\theta(0)&=&h(\theta).
\end{array}
\right.
\end{eqnarray*}

\item More particularly, when $\alpha=1$ and $\Omega_{2} =\mathbb{R}^{n}$, $\Omega_{1} =\mathbb{R}^{m}$, $A= 0$ and $h(x) = x_{0}$, the fractional problem (\ref{principal}) becomes the problems studied by Gwinner \cite{5}; Li, Huang and O'Regan \cite{13}; Liu, Loi and Obukhovskii \cite{15} and Pang and Stewart \cite{19}.
\end{enumerate}

The article is organized as follows. In Section~\ref{Section2}, we present some definitions involving the one and two-parameter Mittag-Leffler functions, as well as several important results which will be used in next sections. In the mean time, we also present some results that help in the discussion of the work, in a special way, some conditions were imposed to guarantee the main result of the paper. Finally, in Section~\ref{Section3}, we will attack the main result of this present paper, i.e., the existence of mild solutions to the fractional problem (\ref{principal}).

\section{Mathematical background: auxiliaries results}\label{Section2}

We start with a generalization of the Cauchy representation for semigroup associated to positive sectorial operators. To this end, we recall that given $\varepsilon > 0$ and $\theta \in(\frac{\pi}{2},\pi)$, the Hankel's path $Ha = Ha(\varepsilon,\theta)$ is the path given by $Ha = Ha_{1} + Ha_{2} - Ha_{3}$ , where $Ha_{i}$ are given by  (e.g., see \cite{Li11,Matar})
\begin{eqnarray}\label{ko}
    Ha_{1}:={t e^{i\theta},\,t\in [\varepsilon,\infty)},\,\, Ha_{2}:={\varepsilon e^{i\theta},\,t\in [-\theta, \theta]},\,\, Ha_{3}:={t e^{-i\theta},\,t\in [\varepsilon,\infty)}
\end{eqnarray}

Let $\alpha\in (0,1)$ and suppose that $\mathcal{A}:D(\mathcal{A})\subset X\rightarrow X$ is a positive sectorial operator. Then, the operators \cite{Li11,Matar}
\begin{equation*}
   \mathbb{E}_{\alpha} (-\xi^{\alpha}\mathcal{A}):= \frac{1}{2\pi i} \int_{Ha} e^{\lambda \xi} \lambda^{\alpha-1} (\lambda+\mathcal{A})^{-1} d\lambda,\,\,\xi\geq 0
\end{equation*}
and
\begin{equation*}
   \mathbb{E}_{\alpha,\alpha} (-\xi^{\alpha}\mathcal{A}):= \frac{t^{1-\alpha}}{2\pi i} \int_{Ha} e^{\lambda \xi} (\lambda+\mathcal{A})^{-1} d\lambda,\,\,\xi\geq 0
\end{equation*}
with $Ha\subset \rho(-\mathcal{A})$ (see Eq.(\ref{ko})), are well defined and $\mathbb{E}_{\alpha} (-\xi^{\alpha} \mathcal{A})$ is strongly continuous, i.e.,
\begin{equation*}
    \lim_{\xi\rightarrow 0+} \left\|\mathbb{E}_{\alpha} (-\xi^{\alpha}\mathcal{A})\theta- \theta) \right\|=0.
\end{equation*}
In the sequel, we can write the operators $\left\{ \mathbb{E}_{\alpha} (-\xi^{\alpha} \mathcal{A})\right\}$ and $\left\{ \mathbb{E}_{\alpha,\alpha} (-\xi^{\alpha} \mathcal{A})\right\}$ as follows \cite{Li11,Matar}
\begin{equation*}
\mathbb{E}_{\alpha}(-\xi^{\alpha} \mathcal{A}) = \int_{0}^{\infty} \Theta_{\alpha}(s) T(s\xi^{\alpha}) ds,\,\,\xi\geq 0,
\end{equation*}
and
\begin{equation*}
\mathbb{E}_{\alpha,\alpha}(-\xi^{\alpha} \mathcal{A}) = \int_{0}^{\infty} \alpha s\Theta_{\alpha}(s) T(s\xi^{\alpha}) ds,\,\,\xi\geq 0,
\end{equation*}
where $\left\{T(\xi):\, \xi\geq 0\right\}$ is the $C_{0}$-semigroup generated by $-\mathcal{A}$.

Consider the topological space $Y$ and $P(Y)$ the collection of all nonempty subsets of $Y$. In this sense, let us consider the following notations:

\begin{itemize}
\item $b(Y)=\left\{D\in P(Y): D\,\,is\,\, bounded \right\}$;

    \item $C(Y):=\left\{D\in P(Y): D\,\,is\,\,closed \right\}$;

    \item $C b(Y):=\left\{ D\in P(Y): D\,\,is\,\,closed\,\,and\,\, bounded\right\}$;

    \item $K(Y):=\left\{D\in P(Y): D\,\,is\,\,compact \right\}$;

    \item $K v(Y):=\left\{D\in P(Y): D\,\,is\,\, compact\,\,and\,\,convex \right\}$;

    \item $P v(Y):=\left\{D\in P(Y): D\,\,is\,\,convex \right\}$;

    \item $C v(Y):=\left\{D\in P(Y): D\,\,is\,\,is \,\,closed\,\, and\,\,convex \right\}$.

    \end{itemize}
    \begin{definition}\label{Definition2.1} Let $X$ and $Y$ be topological spaces and $\Psi:X\rightarrow P(Y)$ be a set valued mapping. We say that $\Psi$ is
    \begin{enumerate}
        \item[(i)] upper semicontinuous (u.s.c.) at $\theta\in X$ if, for every open set $O\subset Y$ with $\Psi(\theta) \subset O$ there exists a neighborhood $N(\theta)$ of $\theta$ such that $\Psi(N(\theta))=:\bigcup_{u\in N(\theta)} \Psi(u)\subset O$. If this holds for every $\theta\in X$, then $\Psi$ is called u.s.c. on $X$.

        \item[(ii)] lower semicontinuous (l.s.c.) at $\theta\in X$ for every open set $O\subset Y$ with $\Psi(\theta)\cap O\neq \emptyset$, there exists a neighborhood $N(\theta)$ of $\theta$ such that $\Psi(y)\cap O\neq \emptyset$ for all $y\in N(\theta)$. If this holds for every $\theta\in X$, then $\Psi$ is called l.s.c. on $X$.

        \item[(iii)] Continuous at $\theta\in X$ if it is both u.s.c. and l.s.c. at $\theta\in X$. If this holds for every $\theta\in X$, then $\Psi$ is called continuous on $X$.
    \end{enumerate}

    \end{definition}

    \begin{definition}\cite{Migorski256,19}\label{Definition2.4} Let $\Omega$ be a Banach space and $I\subset \mathbb{R}$ be an interval.
    \begin{enumerate}
        \item[(i)] $\Psi:I\rightarrow P(\Omega)$ is said to be measurable if for every open subset $O\subset \Omega$ the set $\Psi^{+}(O)$ is measurable in $\mathbb{R}$;

        \item[(ii)] $\Psi:I\rightarrow Cb(\Omega)$ is said to be strongly measurable if there exists a sequence $\left\{\Psi_{n} \right\}_{n=1}^{\infty}$ of step set-valued mappings such that
        \begin{equation*}
            \mathcal{H}(\Psi(\xi),\Psi_{n}(\xi))\rightarrow 0\,\,as\,\,n\rightarrow\infty\,\,for\,\,a.e.\,\,\xi\in I
        \end{equation*}
where $\mathcal{H}$ is the Hausdorff metric on $Cb(\Omega)$.
    \end{enumerate}
    \end{definition}

\begin{definition}\label{Definition2.6} \cite{Migorski256} A set-valued mapping $\Psi: K\subset \Omega \rightarrow P(\Omega)$ is said to be condensing relative to a measure of non compactness (MNC, for short) $\mu$ (or $\mu$-condensing) if for every $\Omega\subset K$ that is not relatively compact we have
\begin{equation*}
    \mu(\Psi(\Omega)) \ngeq\mu(\Omega),
\end{equation*}
where $\mu(\cdot)$ is an MNC in $\Omega$.
\end{definition}

\begin{definition}\label{Definition2.7} Let $K$ be a subset real Banach space $\Omega$ with its dual $\Omega^{*}$.
\begin{enumerate}
    \item[(i)] $\Psi:K\rightarrow \Omega^{*}$ is said to be monotone, if
    \begin{equation*}
        \left<\Psi(v)-\Psi(u),v-u \right>\geq 0,\,\,\forall u,v\in K;
    \end{equation*}

    \item[(ii)] $\phi: \Omega\rightarrow\overline{\mathbb{R}}:=\mathbb{R}\cup \left\{+\infty\right\}$ is said to be proper, convex and lower semicontinuous, if
    \begin{equation*}
        \phi(\lambda u_{1}+(1-\lambda)u_{2}) \leq \lambda\phi(u_1)+(1-\lambda)\phi(u_2)
    \end{equation*}
for all $u_{1},u_{2}\in \Omega$ and $\lambda\in [0,1]$;
\begin{equation*}
    \underset{n\rightarrow\infty}{\lim\inf} \phi(u_{n})\geq \phi(u^*)
\end{equation*}
for $u_{n}\rightarrow u^{*}$ in $\Omega$ and $\left\{u\in \Omega: \phi(u)<+\infty \right\}\neq\emptyset$.
\end{enumerate}
\end{definition}

Consider the formula
\begin{eqnarray}\label{2.1}
    \chi_{T}(\Omega):= \frac{1}{2} \lim_{\delta\rightarrow 0}\,\,\sup_{x\in\Omega}\,\, \max_{|t_{1}-t_{2}|\leq \delta} \left\|x(t_{1})-x(t_{2}) \right\|_{\Omega}
\end{eqnarray}
which will be used later. It is not difficult to check that $\chi_{T}(\Omega)$ satisfies all properties of {\bf Definition 2.5} \cite{principal}.

\begin{theorem}\label{Theorem2.8}\cite{Migorski256} Let $\Theta$ be a convex closed subset of Banach space $\Omega$ and $\Psi:\Theta\rightarrow K v(\Theta)$ be a closed $\mu$-condensing set-valued mapping, where $\mu$ is a non-singular MNC defined on subsets of $\Theta$. Then, $Fix\,\, \Psi:=\left\{\theta\in \Theta:\theta\in \Psi(\theta)\right\}\neq \emptyset$.
\end{theorem}

Consider the following generalized mixed variational inequality:

(VI): Given $w\in \Omega^{*}$, find $u \in K$ such that
\begin{equation}\label{3.1}
    \left< w+G(u),v-u\right>+\phi(v)-\phi(u) \geq0,\,\,\forall v\in K
\end{equation}
where $G:K\rightarrow \Omega^{*}$. We consider the solution set of the mixed variational inequality (\ref{3.1}) by $SOL(K,w+G(\cdot),\phi)$.

Below we will present some extremely important conditions to present some results that served in the discussion of the main results of this paper, i.e., let us consider the following conditions:

$({\bf P_{1}})$ $G:\Omega_{1}\rightarrow \Omega_{1}^{*}$ is monotone and hemi-continuous on $K$;

$({\bf P_{2}})$ $\phi:\Omega_{1}\rightarrow \overline{\mathbb{R}}$ is proper convex and lower semi-continuous;

$({\bf P_{3}})$ there exists $v^{*}\in K\cap D(\phi)$ such that
\begin{equation*}
    \underset{u\in K, ||u||_{\Omega_{1}}\rightarrow\infty}{\lim \infty} \frac{\left< G(u), u-v^{*}\right>+\phi(u)-\phi(v^{*})}{||u||_{\Omega_{1}}}\rightarrow+\infty;
\end{equation*}

$({\bf P_{4}})$ $g:[0,T]\times \Omega_{2}\rightarrow \Omega_{1}^{*}$ is continuous and bounded;

$({\bf P_{5}})$ $\mathcal{B}:[0,T]\times \Omega_{2}\rightarrow \mathscr{L} (\Omega_{1},\Omega_{2})$ satisfies Caratheodory conditions, i.e., $\mathcal{B}(\cdot,\theta):[0,T]\rightarrow\mathscr{L}(\Omega_{1},\Omega_{2})$ is measurable for all $\theta\in \Omega_{2}$ and $\mathcal{B}(\xi,\cdot):\Omega_{2}\rightarrow\mathscr{L}(\Omega_{2},\Omega_{1})$ is continuous for a.e. $\xi\in [0,T]$, where $\mathscr{L}(\Omega_{1},\Omega_{2})$ denotes the class of linear bounded operators from $\Omega_{1}$ to $\Omega_{2}$, and there exist $\rho_\mathcal{B}\in L^{2}([0,T],\mathbb{R}_{+})$ and non-decreasing continuous function $\Upsilon_{\mathcal{B}}:\mathbb{R}_{+}\rightarrow\mathbb{R}_{+}$ such that
\begin{equation*}
    \left\|\mathcal{B}(\xi,\theta) \right\|\leq \rho_\mathcal{B}(\xi) \Upsilon_\mathcal{B}(||\theta||_{\Omega_{2}})
\end{equation*}
for all $(\xi,\theta)\in [0,T]\times \Omega_{2}$;

$({\bf P_{6}})$ Let $h: C([0,T],\Omega_{2})\rightarrow \Omega_{2}$ be an continuous function. Therefore, there exists a non-decreasing continuous function $\Upsilon_{h}$ such that
\begin{equation}\label{4.2}
    \left\|h(\theta) \right\|_{\Omega_{2}}\leq \Upsilon_{h} (||\theta||_{c}),\,\,\forall \theta\in C([0,T],\Omega_{2}),
\end{equation}
where $||\theta||_{c}=\sup_{t\in[0,T]}\|\theta(t)\|_{\Omega_2}$.

$({\bf P_{7}})$ $f(\cdot,\theta): [0,T]\rightarrow \Omega_{2}$ is measurable for all $\theta\in \Omega_{2}$ and there exists $\rho_{f}\in L^{2}([0,T],\mathbb{R}_{+})$
such that for a.e. $\xi\in [0,T]$ we have
\begin{eqnarray}\label{1.11}
\left\{
 \begin{array}{rcl}
 \left\|f(\xi,\theta)-f(\xi,y) \right\|_{\Omega_{2}}&\leq&\rho_{f}(\xi) \left\|\theta-y \right\|_{\Omega_{2}} \,\,\forall \,\theta,y\in \Omega_{2}\notag\\
\left\| f(\xi,0)\right\|_{\Omega_{2}} &\leq&  \rho_{f}(\xi). \notag
\end{array}
\right.
\end{eqnarray}

\begin{theorem}\label{Theorem3.3} \cite{principal} Under the conditions $({\bf P_{1}})-({\bf P_{3}})$ for all $w\in \overline{\mathcal{B}}(n,\Omega_{1}^{*}):=\left\{ w\in \Omega_{1}^{*}: ||u||_{\Omega_{1}^{*}}\leq n\right\}$, there exists  $\Theta_{n}>0$  such that
\begin{equation}\label{3.6}
    ||u||_{\Omega_{1}^{*}}\leq \Theta_{n},\,\, \forall u\in SOL(K,w+G(\cdot), \phi).
\end{equation}
\end{theorem}

Let $g:[0,T]\times \Omega_{2}\rightarrow \Omega_{1}^{*}$. Define the set valued mapping $U: [0,T]\times \Omega_{2}\rightarrow P(K)$ as follows
\begin{eqnarray}\label{3.7}
    U(\xi,\theta):=\left\{u\in K: u\in SOL(K,g(\xi,\theta)+G(\cdot),\phi) \right\}.
\end{eqnarray}

\begin{theorem}\label{Theorem3.4}\cite{principal} Under the conditions $({\bf P_{1}})-({\bf P_{3}})$, if $g:[0,T]\times \Omega_{2}\rightarrow \Omega_{1}^{*}$ is a bounded and continuous function, then we have
\begin{enumerate}
    \item[(i)] $U$ is strongly-weakly u.s.c;

    \item[(ii)] $\xi\mapsto U(\xi,\theta)$ is measurable for every $\theta\in \Omega_{2}$;

    \item[(iii)] for each bounded subset $\Omega$ of $C([0,T],\Omega_{2})$, there exists a constant $\Theta_{\Omega}$ such that
    \begin{equation}\label{3.8}
        \left\|U(\xi,\theta(\xi)) \right\|:= \sup \left\{||u||_{\Omega_{1}}; u\in U(\xi,\theta(\xi)) \right\}\leq \Theta_{\Omega},\,\,\forall \xi\in [0,T]\,\,and\,\, \theta\in \Omega.
    \end{equation}
\end{enumerate}
\end{theorem}

Recall that $\Omega_{1}$ is separable,  using the {\bf Theorem \ref{Theorem3.4}}, it is easy to show that $U(\xi,\theta(\xi))$ admits a measurable selection $l$ such that $l\in L^{\infty}([0,T],\Omega_{1})\subset L^{2}([0,T],\Omega_{1})$ for each $\theta\in C([0,T],\Omega_{2})$. So
\begin{eqnarray}\label{4.1}
    P_{U}(\theta):=\left\{l\in L^{2}([0,T],\Omega_{1}): l(\xi)\subset U(\xi,\theta(\xi))\,\,for\,\,a.e.\,\, \xi\in [0,T] \right\}
\end{eqnarray}
is well-defined for each $\theta\in C([0,T],\Omega_{2})$.

\begin{lemma}\label{Lemma4.1}\cite{principal} Assume that conditions $({\bf P_{1}})-({\bf P_{3}})$ hold and $g:[0,T]\times \Omega_{2}\rightarrow \Omega_{2}^{*}$, is continuous and bounded. Then, set-valued mapping $P_{U}$ is strongly-weakly u.s.c.
\end{lemma}

\section{Existence theorems}\label{Section3}

In this section, we will attack the main contribution of this present paper, i.e., the existence of smooth solutions for fractional differential variational inequality (\ref{principal}).

\begin{theorem}\label{Theorem4.2} Under the conditions $({\bf P_{1}})-({\bf P_{7}})$, if  $\mathbb{E}_{\alpha} (\xi^{\alpha} \mathcal{A})$ is compact for $\xi>0$ and the following inequality holds
\begin{eqnarray}\label{4.3}
    \lim \inf_{k\rightarrow\infty} \left(\Theta_{||g||} ||\rho_\mathcal{B}||_{L^{2}}\frac{\Upsilon_\mathcal{B}(k)}{k}+||\rho_{f}||_{L^{2}}+\frac{\Upsilon_{h}(k)}{k T^{1/2}}  \right)< \frac{1}{\Theta_{\mathcal{A}} T^{1/2}}
\end{eqnarray}
where $\Theta_{\mathcal{A}}:=\sup_{\xi\in [0,T]} \left\|(\xi-s)^{\alpha-1}\mathbb{E}_{\alpha\,\alpha}((\xi-s)^{\alpha}\mathcal{A}) \right\|$ and $\Theta_{||g||}>0$ is the constant (see {\bf Theorem \ref{Theorem3.3}}), then the fractional differential variational inequality {\rm (\ref{principal})} has at least one mild solution $(\theta,u)$.
\end{theorem}

\begin{proof} From the definition of  mild solutions to problem (\ref{principal}), let us consider the mapping $\Gamma: C([0,T],\Omega_{2})\rightarrow P\left(C\left([0,T],\Omega_{2}\right) \right)$ of (\ref{principal}) as follows
\begin{eqnarray}\label{4.4}
    \Gamma(\theta):= \bigg\{&&\mathbb{E}_{\alpha}(\xi^{\alpha} \mathcal{A}) h(\theta)+\int_{0}^{\xi} (\xi-s)^{\alpha-1} \mathbb{E}_{\alpha,\alpha}((\xi-s)^{\alpha}\mathcal{A}) \left(\mathcal{B}(s,\theta(s))l(s)+ f(s,\theta(s)) \right)ds,
    \nonumber\\
    &&\,\,\xi\in[0,T],\,l\in P_{U}(\theta)\bigg\},
\end{eqnarray}
where $P_{U}$ is defined in (\ref{4.1}). Let's prove that $\Gamma$ admits a fixed point in $C([0,T],\Omega_{2})$ and, then, obtain the existence of a solution for problem (\ref{principal}). Indeed, we have the following claims.
\vspace{1.0em}

{\bf Claim 1:} {\it $\Gamma(\theta)\in Kv\left(C([0,T],\Omega_{2}) \right)$ for each $\theta\in C([0,T],\Omega_{2})$.}

Obviously, $\Gamma(\theta)$ has convex values for each $\theta\in C([0,T],\Omega_{2})$ by means of the convexity of $P_{U}(\theta)$. If remains to demonstrate the compactness of $\Gamma(\theta)$ for $\theta\in C([0,T],\Omega_{2})$. For any $y\in \Gamma(\theta)$, we can find $l\in P_{U}(\theta)$ such that
\begin{equation*}
    y(\xi)=\mathbb{E}_{\alpha} (\xi^{\alpha}\mathcal{A}) h(\theta)+\int_{0}^{\xi}(\xi-s)^{\alpha-1}\mathbb{E}_{\alpha,\alpha}(\mathcal{A}(\xi-s)^{\alpha}) (\mathcal{B}(s,\theta(s))l(s)+f(s,\theta(s)))ds
\end{equation*}
for all $\xi\in[0,T]$. From H\"older inequality, it yields
\begin{eqnarray*}
    &&\quad\left\|y(\xi) \right\|_{\Omega_{2}}\\
    &&=\left\|\mathbb{E}_{\alpha} (\xi^{\alpha}\mathcal{A}) h(\theta)+ \int_{0}^{\xi} (\xi-s)^{\alpha-1}\mathbb{E}_{\alpha,\alpha}(\mathcal{A}(\xi-s)^{\alpha}) (\mathcal{B}(s,\theta(s))l(s)+f(s,\theta(s)))\right\|_{\Omega_{2}}\notag\\
    &&\leq\int_{0}^{\xi} \left\|(\xi-s)^{\alpha-1}\mathbb{E}_{\alpha,\alpha} (\mathcal{A}(\xi-s)^{\alpha}) \right\|_{\Omega_{2}} \left\|\mathcal{B}(s,\theta(s))l(s)+f(s,\theta(s)) \right\|_{\Omega_{2}}ds
    \notag\\
    &&\quad+\left\|\mathbb{E}_{\alpha} (\xi^{\alpha} \mathcal{A}) h(\theta) \right\|_{\Omega_{2}}\\
    &&\leq\int_{0}^{\xi} \left\|(\xi-s)^{\alpha-1}\mathbb{E}_{\alpha,\alpha} (\mathcal{A}(\xi-s)^{\alpha}) \right\|_{\Omega_{2}} \left(\left\|\mathcal{B}(s,\theta(s))l(s)\right\|_{\Omega_{2}}+\left\|f(s,\theta(s)) \right\|_{\Omega_{2}}\right)ds\\
    &&\quad+\left\|\mathbb{E}_{\alpha} (\xi^{\alpha} \mathcal{A}) h(\theta) \right\|_{\Omega_{2}}
    \notag\\
    &&\leq \sup_{\xi\in [0,T]}\left\|\mathbb{E}_{\alpha} (\mathcal{A}\xi^{\alpha}) \right\|_{\Omega_{2}}\int_{0}^{\xi} \left(\left\|\mathcal{B}(s,\theta(s))l(s)\right\|_{\Omega_{2}}+\left\|f(s,\theta(s)) \right\|_{\Omega_{2}}\right)ds\notag\\
    &&\quad+\sup_{\xi\in[0,T]}\left\|\mathbb{E}_{\alpha} (\xi^{\alpha} \mathcal{A}) \right\|_{\Omega_{2}} ||h(\theta)||\\
    &&= \Theta_{\mathcal{A}} \Upsilon(||\theta||_{c})+ \Theta_{\mathcal{A}} \int_{0}^{\xi} \left(\left\|\mathcal{B}(s,\theta(s))l(s)\right\|_{\Omega_{2}}+\left\|f(s,\theta(s)) \right\|_{\Omega_{2}}\right)ds\notag\\
    &&\leq \Theta_{\mathcal{A}} \Upsilon(||\theta||_{c})+ \Theta_{\mathcal{A}} \int_{0}^{\xi} \left(\Theta_{||g||} \rho_\mathcal{B}(s)\Upsilon_\mathcal{B}(||\theta(s)||_{\Omega_{2}})+\rho_{f}(s) (1+||\theta(s)||_{\Omega_{2}}) \right)ds
    \notag\\
    &&\leq \Theta_{\mathcal{A}} \Upsilon(||\theta||_{c})+ \Theta_{\mathcal{A}} \int_{0}^{\xi} \left(\Theta_{||g||} \rho_\mathcal{B}(s)\Upsilon_\mathcal{B}(||\theta||_{c})+\rho_{f}(s) (1+||\theta||_{c}) \right)ds
    \notag\\
    &&\leq \Theta_{\mathcal{A}} \Upsilon(||\theta||_{c})+ \Theta_{\mathcal{A}}  \Theta_{||g||} ||\rho_\mathcal{B}||_{L^{2}}\Upsilon_\mathcal{B}(||\theta||_{c}) T^{1/2}+\Theta_{\mathcal{A}}||\rho_{f}||_{L^{2}} (1+||\theta||_{c}) T^{1/2} .
\end{eqnarray*}
So, for $\theta\in C([0,T],\Omega_{2})$, the mapping $\Gamma(\theta)$ is bounded in $ C([0,T],\Omega_{2})$. On the other hand, for $\theta\in C([0,T],\Omega_{2})$, let's prove that $\Gamma(\theta)$ is a familiar of equicontinuous functions. To this end, let us distinguish the following two cases. Let  $\varepsilon_{0}$  be small enough.

    {\bf Case 1:} Assume that $\xi_{1}=0$ and $0<\xi_{2}\leq \varepsilon_{0}$.  Using Holder inequality and the compactness of $\mathbb{E_{\alpha}}(\mathcal{A} \xi^{\alpha})$, one has
\begin{eqnarray}
&&\quad\left\| y(\xi_{2})-y(\xi_1)\right\|_{\Omega_{2}}
\\&=& \left\|\mathbb{E}_{\alpha}(\mathcal{A}\xi_{2}^{\alpha})h(\theta)+ \int_{0}^{\xi_{2}} (\xi_{2}-s)^{\alpha-1}\mathbb{E}_{\alpha,\alpha}(\mathcal{A}(\xi_{2}-s)^{\alpha}) (\mathcal{B}(s,\theta(s))l(s)+f(s,\theta(s)))ds-h(\theta)\right\|_{\Omega_{2}}\notag\\
&&\leq \left\|(\mathbb{E} (\mathcal{A}\xi_{2}^{\alpha})-I)h(\theta) \right\|_{\Omega_{2}}+\int_{0}^{\xi_2} (\xi_{2}-s)^{\alpha-1}\mathbb{E}_{\alpha,\alpha}(\mathcal{A}(\xi_{2}-s)^{\alpha}) \left\| \mathcal{B}(s,\theta(s))l(s)+f(s,\theta(s))\right\|_{\Omega_{2}}ds\notag\\
&&\leq\left\|(\mathbb{E}_{\alpha} (\mathcal{A}\xi_{2}^{\alpha})- I) h(\theta)\right\|_{\Omega_{2}} + \Theta_{\mathcal{A}} \int_{0}^{\xi_{2}} \left(\Theta_{||g||} \rho_\mathcal{B}(s) \Upsilon_\mathcal{B}(||\theta(s)||_{\Omega_{2}})+\rho_{f(s)} (1+||\theta(s)||_{\Omega_{2}})   \right)ds\notag\\
&&\leq \left\|(\mathbb{E}_{\alpha} (\mathcal{A} \xi_{2}^{\alpha})-I))h(\theta) \right\|_{c}+ \Theta_{\mathcal{A}} \int^{\xi_2}_{0} \Theta_{||g||}\rho_\mathcal{B}(s)\Upsilon(||\theta(s)||_{c}) ds+\Theta_{\mathcal{A}}\int^{\xi_2}_{0}\rho_{f(s)} (1+||\theta(s)||_{c})ds\notag\\
&&\leq \left\|(\mathbb{E}_{\alpha} (\mathcal{A} \xi_{2}^{\alpha})-I))h(\theta) \right\|_{c}+ \Theta_{\mathcal{A}} \xi_{2}^{1/2} \Theta_{||g||}||\rho_\mathcal{B}||_{L^2}\Upsilon(||\theta||_{c}) +\Theta_{\mathcal{A}}\xi_{2}^{1/2}||\rho_{f}||_{L^2} (1+||\theta||_{c})\notag\\
&&\leq \left\|(\mathbb{E}_{\alpha} (\mathcal{A} \xi_{2}^{\alpha})-I)) \right\|_{c} \Upsilon_{h}\left(\left\|\theta \right\|_{c}\right) + \Theta_{\mathcal{A}} \xi_{2}^{1/2} \Theta_{||g||}||\rho_\mathcal{B}||_{L^2}\Upsilon(||\theta||_{c}) +\Theta_{\mathcal{A}}\xi_{2}^{1/2}||\rho_{f}||_{L^2} (1+||\theta||_{c}).\notag
\end{eqnarray}
This implies that $\left\| y(\xi_{2})-y(\xi_1)\right\|_{\Omega_{2}}\to 0$
 as $0<\xi_{2}\leq \varepsilon_{0}\rightarrow 0$.
\vspace{1.0em}

{\bf Case 2:} If $\frac{\varepsilon_0}{2}\leq \xi_{1} \leq \xi_{2}\leq T$, then we have
\begin{eqnarray}
&&\quad\left\| y(\xi_{2})-y(\xi_1)\right\|_{\Omega_{2}}\\
&&\leq  \left\|\mathbb{E}_{\alpha} (\xi_{2}^{\alpha} \mathcal{A})- \mathbb{E}_{\alpha} (\xi_{1}^{\alpha} \mathcal{A})\right\|_{\Omega_{2}} |||h(\theta)||_{\Omega_{2}}
+\left\| \int_{0}^{\xi_2} (\xi_{2}-s)^{\alpha-1}\mathbb{E}_{\alpha,\alpha}(\mathcal{A}(\xi_{2}-s)^{\alpha}) \left(\mathcal{B}(s,\theta(s))l(s)+f(s,\theta(s))\right)ds\right.\notag\\
&&\quad\left.-\int_{0}^{\xi_1} (\xi_{1}-s)^{\alpha-1}\mathbb{E}_{\alpha,\alpha}(\mathcal{A}(\xi_{1}-s)^{\alpha})  \left(\mathcal{B}(s,\theta(s))l(s)+f(s,\theta(s))\right)ds\right\|_{\Omega_{2}}\notag\\
 &&\leq
\left\| \int_{\xi_1}^{\xi_2} (\xi_{2}-s)^{\alpha-1}\mathbb{E}_{\alpha,\alpha}(\mathcal{A}(\xi_{2}-s)^{\alpha}) \left(\mathcal{B}(s,\theta(s))l(s)+f(s,\theta(s))\right)ds\right\|_{\Omega_{2}}+\left\|\mathbb{E}_{\alpha} (\xi_{2}^{\alpha} \mathcal{A})- \mathbb{E}_{\alpha} (\xi_{1}^{\alpha} \mathcal{A})\right\|_{\Omega_{2}}\Upsilon_{h}(||\theta||_c) \notag\\
&&\quad +\left\| \int_{0}^{\xi_1} \left( (\xi_{2}-s)^{\alpha-1}\mathbb{E}_{\alpha,\alpha}(\mathcal{A}(\xi_{2}-s)^{\alpha})-(\xi_{1}-s)^{\alpha-1}\mathbb{E}_{\alpha}(\mathcal{A}(\xi_{1}-s)^{\alpha})\right) \left(\mathcal{B}(s,\theta(s))l(s)+f(s,\theta(s))\right)ds\right\|_{\Omega_{2}}\notag\\
&&= :{\bf Q_{1}}+{\bf Q_{2}}+{\bf Q_{3}}.\notag
\end{eqnarray}
Let's analyze ${\bf Q_{1}},{\bf Q_{2}}$ and ${\bf Q_{3}}$. Using the conditions $({\bf P_{5}})-({\bf P_{7}})$ and  H\"older inequality, one has
\begin{eqnarray}\label{4.5}
&&\quad{\bf Q_{1}}\notag\\
&&= \left\| \int_{\xi_1}^{\xi_2}(\xi_{2}-s)^{\alpha-1}\mathbb{E}_{\alpha,\alpha}(\mathcal{A}(\xi_{2}-s)^{\alpha}) \left(\mathcal{B}(s,\theta(s))l(s)+f(s,\theta(s))\right)ds\right\|_{\Omega_{2}}\notag\\
&&\leq \Theta_{\mathcal{A}} \int_{\xi_1}^{\xi_2} \left(\left\|  \left(\mathcal{B}(s,\theta(s))l(s)\right\|+\left\|f(s,\theta(s))\right\|_{\Omega_{2}}\right)\right)ds\notag\\
&&\leq \Theta_{\mathcal{A}} \int_{\xi_1}^{\xi_{2}} \left( \Theta_{||g||}\rho_\mathcal{B}(s)\Upsilon_\mathcal{B}(||\theta(s)||_{\Omega_{2}})+\rho_{f}(s) (1+||\theta(s)||_{\Omega_{2}}) \right)ds\notag\\
&&\leq \Theta_{\mathcal{A}} (\xi_{2}-\xi_{1})^{1/2} \left(\Theta_{||g||}\Upsilon(||\theta||_{c})||\rho_\mathcal{B}||_{L^{2}}+||\rho_f||_{L^2}(1+||\theta||_c)      \right) \rightarrow 0
\end{eqnarray}
as $|\xi_{2}-\xi_{1}|\rightarrow 0$.

The continuity and compactness of  $\mathbb{E}_{\alpha}(\mathcal{A}\xi^{\alpha})$ implies that for $\varepsilon_{0}/2\leq \xi_{1}<\xi_{2} \leq T$, it holds
\begin{eqnarray}\label{4.6}
    {\bf Q_{2}} = \left\|\mathbb{E}_{\alpha}(\mathcal{A} \xi_{2}^{\alpha})-\mathbb{E}_{\alpha}(\mathcal{A} \xi_{1}^{\alpha}) \right\| \left\| h(\theta)\right\| _{\Omega_{2}} \leq \left\|\mathbb{E}_{\alpha} (\mathcal{A} \xi_{2}^{\alpha}) - \mathbb{E}_{\alpha}(\mathcal{A} \xi_{1}^{\alpha}) \right\| \Upsilon_{h} (||\theta||_{c})\rightarrow 0
\end{eqnarray}
as $|\xi_{2}-\xi_{1}|\rightarrow 0$.

On the other hand, for $\delta>0$, it yields
\begin{eqnarray}\label{4.7}
    &&\quad{\bf Q_{3}}\notag\\
    &&\leq \left\|\int_{\xi_{1}-\delta}^{\xi_{1}}\left((\xi_{2}-s)^{\alpha-1}\mathbb{E}_{\alpha,\alpha}(\mathcal{A}(\xi_{2}-s)^{\alpha})-(\xi_{1}-s)^{\alpha-1}\mathbb{E}_{\alpha,\alpha}(\mathcal{A}(\xi_{1}-s)^{\alpha})  \right) (\mathcal{B}(s,\theta(s))l(s)+f(s,\theta(s)))ds\right\|_{\Omega_{2}}\notag\\
    &&\quad+ \left\|\int_{0}^{\xi_{1}-\delta}\left(\xi_{2}-s)^{\alpha-1}(\mathbb{E}_{\alpha,\alpha}(\mathcal{A}(\xi_{2}-s)^{\alpha})-(\xi_{1}-s)^{\alpha-1}\mathbb{E}_{\alpha,\alpha}(\mathcal{A}(\xi_{1}-s)^{\alpha})  \right) \left(\mathcal{B}(s,\theta(s))l(s)+f(s,\theta(s))\right)ds\right\|_{\Omega_{2}}\notag\\
    &&\leq  2\Theta_{\mathcal{A}} \int^{\xi_{1}}_{\xi_{1}-\delta}\left(\left\|\mathcal{B}(s,\theta(s))l(s)\right\|_{\Omega_{2}}+ \left\|f(s,\theta(s))\right\|_{\Omega_{2}}\right)ds\notag\\
    &&\quad+ \underset{s\in [0,\xi_{1}-\delta]}{\sup} \left\| (\xi_{2}-s)^{\alpha-1}\left( \mathbb{E}_{\alpha,\alpha} (\mathcal{A} (\xi_{2}-s)^{\alpha})- \mathbb{E}_{\alpha,\alpha} (\mathcal{A} (\xi_{1}-s)^{\alpha})\right)\right\|\notag\\
    &&\quad\times \int_{0}^{\xi_{1}-\delta} \left(\left\|\mathcal{B}(s,\theta(s))l(s) \right\|_{\Omega_{2}}+ \left\|f(s,\theta(s)) \right\|_{\Omega_{2}} \right)ds\\
    &&\leq 2 \delta^{1/2} \Theta_{\mathcal{A}} \left(\Theta_{||g||} \Upsilon_\mathcal{B}(||\theta||_c) ||\rho_\mathcal{B}||_{L^{2}}+||\rho_{f}||_{L^{2}} \left( 1+||\theta||_{c}\right) \right)\notag\\
    &&\quad+\underset{s\in [0,\xi_{1}-\delta]}{\sup} \left\|(\xi_{2}-s)^{\alpha-1}\left( \mathbb{E}_{\alpha,\alpha} (\mathcal{A} (\xi_{2}-s)^{\alpha})- \mathbb{E}_{\alpha,\alpha} (\mathcal{A} (\xi_{1}-s)^{\alpha})\right) \right\| (\xi_{1}-\delta)^{1/2}\notag\\ &&\quad\times\left(\Theta_{||g||} \Upsilon_\mathcal{B} (||\theta||_{c}) ||\rho_\mathcal{B}||_{L^{2}}+||\rho_{f}||_{L^{2}}   (1+||\theta||_{c})\right)\rightarrow 0\notag
\end{eqnarray}
as $|\xi_{2}-\xi_{1}|\rightarrow 0$ and $\delta\rightarrow 0$.

From (\ref{4.5})-(\ref{4.7}), it follows that
\begin{eqnarray*}
    \left\|y(\xi_2)-y(\xi_1) \right\|_{\Omega_{2}}\rightarrow 0,\,\, as\,\, |\xi_{2}-\xi_{1}|\rightarrow 0.
\end{eqnarray*}

So, the set $\Gamma(\theta)$ is equicontinuous for every $\theta\in C([0,T],\Omega_{2})$. Using the Arzela-Ascoli theorem \cite{27}, we have that $\Gamma(\theta)$ is relatively compact for each $\theta\in C([0,T],\Omega_{2})$. Now let's verify that $\Gamma(\theta)$ is closed in $C([0,T],\Omega_{2})$ for each $\theta\in C([0,T],\Omega_{2})$. Consider $\left\{y_n\right\}\subset \Gamma(\theta)$ a sequence such that $y_{n}\rightarrow y^{*}$ in $C([0,T],\Omega_{2})$ as $n\rightarrow \infty$. So, there exists a sequence $\left\{ l_{n}\right\}\subset P_{U}(\theta)$ such that
\begin{eqnarray*}
    y_{n}(\xi)=\mathbb{E}_{\alpha}(\mathcal{A} \xi^{\alpha}) h(\theta)+\int_{0}^{\xi}(\xi-s)^{\alpha-1} \mathbb{E}_{\alpha,\alpha}(\mathcal{A}(\xi-s)^{\alpha}) \left( \mathcal{B}(s,\theta(s))l_{n}(s)+f(s,\theta(s))\right)ds,\,\, \xi\in [0,T].
\end{eqnarray*}
Using {\bf Theorem \ref{Theorem3.4}} (3), it follows that  sequence $\left\{l_{n} \right\}\subset L^{2}([0,T],\Omega_{2})$ is weakly relatively compact. Without loss of generality, we may assume $l_{n}\rightharpoonup l^{*}\in P_{U}(\theta)$ in $L^{2}([0,T],\Omega_{1})$. However, by Mazur's lemma and the the upper semi-continuity of $P_{U}$, it is not difficult to see that $l^{*}\in P_{U}(\theta)$. On the other hand, the compactness of $\mathbb{E}_{\alpha}(\mathcal{A} \xi^{\alpha})$ for $\xi>0$ reveals that
\begin{eqnarray*}
    y^{*}(\xi)= \mathbb{E}_{\alpha}(\mathcal{A} \xi^{\alpha})h(\theta)+\int_{0}^{\xi} \mathbb{E}_{\alpha}(\mathcal{A} (\xi-s)^{\alpha}) (\mathcal{B}(s,\theta(s))l^{*}(s)+f(s,\theta(s)))ds,\,\,\forall \xi\in [0,T].
\end{eqnarray*}
This means $y^{*}\in \Gamma(\theta)$, i.e., $\Gamma(\theta)\in Kv(C([0,T],\Omega_{2}))$.
\vspace{1.0em}

{\bf Claim 2.} {\it $\Gamma$ is a closed mapping.}
\vspace{1.0em}

Let $\theta_{n}\rightarrow \theta^{*}$ and $y_{n}\rightarrow y^{*}$ in $C([0,T], \Omega_{2})$ with $y_{n}\in \Gamma(\theta_{n})$ for each $n\in\mathbb{N}$. Now let's prove that $y^{*}\in \Gamma(\theta)$. Using the definition of $\Gamma$, we can choose $l_{n}\in P_{U}(\theta_{n})$ with $n\in\mathbb{N}$ such that
\begin{eqnarray}\label{4.8}
    y_{n}(\xi)= \mathbb{E}_{\alpha} (\mathcal{A} \xi^{\alpha}) h(\theta_{n})+\int^{\xi}_{0} (\xi-s)^{\alpha-1}\mathbb{E}_{\alpha,\alpha}(\mathcal{A} (\xi-s)^{\alpha}) \left( \mathcal{B}(s,\theta_{n}(s))l_{n}(s)+f(s,\theta_{n}(s)\right)ds
\end{eqnarray}
for all $\xi\in [0,T]$. Without loss of generality, using {\bf Theorem \ref{Theorem3.4}} and {\bf Lemma \ref{Lemma4.1}}, we may assume that $l_{n}\rightarrow l^{*}\in P_{U}(\theta^{*})$. In addition, using the condition $({\bf P_{7}})$, we have that $f(\cdot,\theta_{n}(\cdot))\rightarrow f(\cdot,\theta^{*}(\cdot))$ in $L^{2}([0,T],\Omega_{2})$. Using the continuity of $h$ and $\mathcal{B}(\xi,\cdot)$ and the compactness of $\mathbb{E}_{\alpha}(\mathcal{A} \xi^{\alpha})$ for $\xi>0$ in (\ref{4.8}), one has
\begin{eqnarray*}
    y^{*}(\xi)= \mathbb{E}_{\alpha}(\mathcal{A}\xi^{\alpha}) h(\theta^{*}) + \int_{0}^{\xi} (\xi-s)^{\alpha-1}\mathbb{E}_{\alpha,\alpha}((\xi-s)^{\alpha} \mathcal{A})\left(\mathcal{B}(s,\theta^{*}(s))l^{*}(s)+ f(s,\theta^{*}(s)) \right)ds,\,\, \xi\in [0,T]
\end{eqnarray*}
and $l^{*}\in P_{U}(\theta^{*})$. Therefore, we have $y^{*}\in \Gamma (\theta^{*})$.
\vspace{1.0em}

{\bf Claim 3.} {\it $\Gamma$ is $\chi_{T}$-condensing.}
\vspace{1.0em}

It suffices to prove that $\chi_{D}\nleq \chi_{T}(\Gamma(D))$ for each $D\in b(C([0,T],\Omega_{2}))$ being not relatively compact in $C([0,T],\Omega_{2})$. Since $D\subset C([0,T],\Omega_{2})$ is a bounded set, $\Gamma(D)$ is relatively compact, i.e., $\chi_{T}(\Gamma(D))=0$. So, $\chi_{T}(D)\leq \chi_{T}(\Gamma(D))=0$ implies that $D$ is relatively compact. Finally, using the regularity of $\chi_{T}$ we have that $\Gamma$ is $\chi_{T}$-condensing.

\vspace{1.5em}

{\bf Claim 4.} {\it There exists a constant $\Theta_{R}>0$ such that}
\begin{eqnarray*}
    \Gamma(\overline{\mathcal{B}}_{\Theta_R})\subset \overline{\mathcal{B}}_{\Theta_R}:=\left\{\theta\in C([0,T],\Omega_{2}); ||\theta||_{c}\leq \Theta_{R}\subset C([0,T];\Omega_{2})\right\}.
\end{eqnarray*}

The proof will be discussed by contradiction. In this sense, we assume that there are two subsequences
$\left\{\theta_{k}\right\}$ and $\left\{y_{k}\right\}$   such that  $\left\|\theta_k \right\|_{C([0,T],\Omega_{2})}\leq k$, $y_{k}\in \Gamma(\theta_k)$ and $\left\|y_k \right\|_{C([0,T],\Omega_{2})}>k$. So, there exists $l_k\in P_{U}(\theta_k)$ such that
\begin{eqnarray*}
    y_{k}(\xi) = \mathbb{E}_{\alpha} (\mathcal{A} \xi^{\alpha}) h(\theta_k)+\int_{0}^{\xi}(\xi-s)^{\alpha-1} \mathbb{E}_{\alpha,\alpha}(\mathcal{A}(\xi-s)^{\alpha})\left(\mathcal{B}(s,\theta(s))l_{k}(s)+f(s,x_{k}(s)) \right)ds,\,\,\xi\in [0,T].
\end{eqnarray*}
Then, for every $\xi\in [0,T]$, one has
\begin{eqnarray*}
    &&\quad\left\|y_{k} \right\|_{\Omega_{2}}\\
    &&\leq \left\|\mathbb{E}_{\alpha} (\mathcal{A} \xi^{\alpha}) h(\theta_k) \right\|_{\Omega_{2}}+ \left\|\int_{0}^{\xi} (\xi-s)^{\alpha-1}\mathbb{E}_{\alpha} (\mathcal{A} (\xi-s)^{\alpha}) \left(\mathcal{B}(s,\theta(s))l_{k}(s)+f(s,x_{k}(s)) \right) ds\right\|_{\Omega_{2}}\notag\\
    &&\leq\Theta_{\mathcal{A}}\int_{0}^{\xi} \left[\Theta_{||g||} \rho_\mathcal{B}(s) \Upsilon_\mathcal{B}(||\theta_k(s)||_{\Omega_{2}}) +\rho_{f}(s) (1+||\theta_{k}(s)||_{\Omega_{2}}) \right]ds+ \Theta_{\mathcal{A}} \Upsilon_{h}(k)\notag\\
    &&\leq \Theta_{\mathcal{A}} \left[\int_{0}^{\xi} \left(\Theta_{||g||} \Upsilon_\mathcal{B}(s) \Upsilon_\mathcal{B}(k)+\rho_{f}(s)(1+k) \right) ds  \right]+\Theta_{\mathcal{A}}\Upsilon_{h}(k)\notag\\
    &&\leq \Theta_{\mathcal{A}} T^{1/2} \left(\Theta_{||g||} \Upsilon_\mathcal{B}(k) ||\rho_B||_{L^2}+ (1+k)||\rho_{f}||_{L^2} \right)+\Theta_{\mathcal{A}} \Upsilon_{h}(k),
\end{eqnarray*}
which implies by (\ref{4.3}) that
\begin{eqnarray*}
    &&\,1\leq \underset{k\rightarrow\infty}{\lim\inf} \frac{||y_k||_{C([0,T],\Omega_{2})}}{k}\notag\\
    &&\quad\leq \underset{k\rightarrow\infty}{\lim\inf} \left(\Theta_{\mathcal{A}} T^{1/2} \Theta_{||g||} ||\rho_\mathcal{B}||_{L^2} \frac{\Upsilon_\mathcal{B}(k)}{k} +\Theta_{\mathcal{A}} \frac{\Upsilon_{h}(k)}{k}+\Theta_{\mathcal{A}}T^{1/2} ||\rho_{f}||_{L^2}\right)\notag\\
    &&\quad<1.
\end{eqnarray*}
This gets a contradiction. Therefore, there exists $\Theta_{R}$ such that {\bf Claim} holds.  All conditions of {\bf Theorem \ref{Theorem2.8}} are verified. Therefore, this implies $Fix\,\,\Gamma \neq \emptyset$ in $\overline{\mathcal{B}}_{\Theta_R}$, namely, fractional differential variational inequality  (\ref{principal}) has at least one mild solution $(\theta,u)$.
\end{proof}

\begin{theorem} Under the conditions $({\bf P_{1}})-({\bf P_{7}})$, if $\mathbb{E}_{\alpha}(\mathcal{A} \xi^{\alpha})$ is compact for $\xi>0$ and
\begin{eqnarray}\label{4.9}
    \underset{k\rightarrow\infty}{\lim} \frac{\Upsilon_{h}(k)}{k}=0,\,\,and\,\,\Upsilon_\mathcal{B}(k)=(1+k)
\end{eqnarray}
then problem {\rm(\ref{principal})} has at least one mild solution $(\theta,u)$.
\end{theorem}

\begin{proof}
Firstly, consider an equivalent norm in space $C([0,T];\Omega_{2})$ given by 
\begin{eqnarray*}
    ||\theta||_{*}:=\underset{\xi\in[0,T]}{\max} \mathbb{E}_{\alpha}(-L \xi^{\alpha}) ||\theta(\xi)||_{\Omega_{2}},\,\,\forall \theta\in C([0,T],\Omega_{2})
\end{eqnarray*}
where $L>0$ is  such that
\begin{eqnarray}\label{4.10}
    \Theta_{\mathcal{A}}\int_{0}^{\xi} \mathbb{E}_{\alpha} (-L (\xi-s)^{\alpha}) \left(\Theta_{||g||} \rho_\mathcal{B}(s)+\rho_{f}(s) \right)ds<1,\,\,for\,\,all\,\,\xi\in[0,T].
\end{eqnarray}

If {\bf Claim 4} is not true, suppose that for each $k>0$ there exist two sequences $\left\{x_k\right\}$ and $\left\{y_k\right\}$ with $y_{k}\in \Gamma(\theta_k)$ such that $||\theta_k||_{C([0,T],\Omega_{2})}\leq k$ and $||y_k||_{C([0,T],\Omega_{2})}> k$. So, there exists $l_{k}\in P_{U}(\theta_k)$ such that
\begin{eqnarray}\label{4.11}
    &&\quad\mathbb{E}_{\alpha}(-L\xi^{\alpha}) ||y_k(\xi)||_{\Omega_{2}}\\
    &&\leq\mathbb{E}_{\alpha}(-L \xi^{\alpha})\int_{0}^{\xi} \left\| (\xi-s)^{\alpha-1}\mathbb{E}_{\alpha,\alpha}(\mathcal{A} (\xi-s)^{\alpha})\right\| \left\|\mathcal{B}(s,x_{k}(s))l_{k}(s)+f(s,x_{k}(s)) \right\|_{\Omega_{2}}ds\notag\\
    &&\quad+\mathbb{E}_{\alpha} (-L \xi^{\alpha}) \Theta_{\mathcal{A}} \Upsilon_{h}(k)\notag\\
&&\leq \Theta_{\mathcal{A}} \mathbb{E}_{\alpha}(-L\xi^{\alpha})\Upsilon_{h}(k)+\Theta_{\mathcal{A}} \mathbb{E}_{\alpha}(-L\xi^{\alpha}) \int_{0}^{\xi} \left(\Theta_{||g||} \rho_\mathcal{B}(s)+\rho_{f}(s) \right)(1+||\theta_{k}(s)||_{\Omega_{2}})ds \notag\\
&&\leq \Theta_{\mathcal{A}}||\theta_{k}||_{*}\int^{\xi}_{0}(\xi-s)^{\alpha-1}\mathbb{E}_{\alpha,\alpha}(-L (\xi-s)^{\alpha}) (\Theta_{||g||} \rho_\mathcal{B}(s)+\rho_{f}(s)) ds\notag\\
&&\quad+\Theta_{\mathcal{A}} \left(\Upsilon_{h}(k)+\Theta_{||g||}||\rho_\mathcal{B}||_{L^1}+||\rho_f||_{L^1} \right)\notag\\
&&\leq \Theta_{\mathcal{A}}k
\int^{\xi}_{0}(\xi-s)^{\alpha-1}\mathbb{E}_{\alpha,\alpha}(-L (\xi-s)^{\alpha}) (\Theta_{||g||} \rho_\mathcal{B}(s)+\rho_{f}(s)) ds\notag\\
&&\quad+\Theta_{\mathcal{A}} \left(\Upsilon_{h}(k)+\Theta_{||g||} ||\rho_\mathcal{B}||_{L^1}+ ||\rho_f||_{L^1} \right).\notag
\end{eqnarray}
Because of $||y_{k}||_{*}>k$ for each $k>0$, one yields from (\ref{4.9})-(\ref{4.11}) that
\begin{eqnarray}
    &&\,1\leq  \underset{k\rightarrow\infty}{ \lim\sup} \frac{||y_k||_{*}}{k} \notag\\
    &&\quad\leq \underset{k\rightarrow\infty}{ \lim \sup} \frac{\Theta_{\mathcal{A}} \Upsilon_{h}(k)}{k}+\Theta_{\mathcal{A}} \int_{0}^{\xi}(\xi-s)^{\alpha-1} \mathbb{E}_{\alpha,\alpha}(-L (\xi-s)^{\alpha}) \left(\Theta_{||g||} \rho_\mathcal{B}(s)+\rho_{f}(s) \right)ds\notag\\
    &&\quad\leq  \Theta_{\mathcal{A}} \int_{0}^{\xi}\mathbb{E}_{\alpha}(-L (\xi-s)^{\alpha}) (\Theta_{||g||}\rho_\mathcal{B}(s)+\rho_{f}(s))ds\notag\\
    &&\quad<1.
\end{eqnarray}
This leads to a contradiction. Hence, the {\bf Claim 4} holds. It follows from {\bf Theorem \ref{Theorem4.2}} that problem (\ref{principal}) has at least one mild solution $(\theta,u)$.
\end{proof}

\subsection*{Data availability statement}
Data sharing not applicable to this article as no data sets were generated or analysed during the current study.

\subsection*{Conflict of interests} There is no conflict of interests.

\section*{Acknowledgment}

The authors are very grateful to the anonymous reviewers for their useful comments that led to improvement of the manuscript. This project has received funding from the  Natural Science Foundation of Guangxi Grant Nos. 2021GXNSFFA196004 and GKAD23026237, the NNSF of China Grant No. 12001478, the China Postdoctoral Science Foundation funded project No. 2022M721560,  the Startup Project of Doctor Scientific Research of Yulin Normal University No. G2023ZK13,  and the European Union's Horizon 2020 Research and Innovation Programme under the Marie Sklodowska-Curie grant agreement No. 823731 CONMECH. It is also supported by the project cooperation between Guangxi Normal University and Yulin Normal University.

\end{document}